\documentclass{amsart}
\usepackage{graphicx, amssymb, amsmath,amsthm,mathtools,verbatim} 
\usepackage[T1]{fontenc} 
\usepackage[utf8]{inputenc} 
\usepackage[sc]{mathpazo}
\usepackage[small]{eulervm} 
\usepackage[dvipsnames]{xcolor}
\usepackage{enumitem}
\usepackage[shadow,loadshadowlibrary]{todonotes}

\DeclarePairedDelimiter{\inner}{\langle}{\rangle}

\newcommand{\ZZ}{\mathbb{Z}}
\newcommand{\RR}{\mathbb{R}}

\newcommand{\LO}{\mathrm{LO}}
\newcommand{\ooo}{\mathfrak{o}}

\newtheorem{theorem}{Theorem}[section]
\newtheorem{lemma}[theorem]{Lemma}
\newtheorem{corollary}[theorem]{Corollary}
\newtheorem{proposition}[theorem]{Proposition}
\newtheorem{question}[theorem]{Question}

\newtheorem{conjecture}[theorem]{Conjecture}
\theoremstyle{definition}
\newtheorem{definition}[theorem]{Definition}

\newtheorem{example}[theorem]{Example}
\theoremstyle{remark}

\title[Non-left orderability of Dehn fillings]{A combinatorial approach to obstructing left-orderability of Dehn fillings}

\author[A.~Clay, J.~Lu, and T.~Staruch]{Adam Clay, Junyu Lu, and Tjay Staruch}

\address{Department of Mathematics, 313 St. Paul's College, 
University of Manitoba, Winnipeg, MB, R3T 2M6}

\email{Adam.Clay@umanitoba.ca}
\email{luj9@myumanitoba.ca}
\email{starucht@myumanitoba.ca}
 \subjclass[2010]{57M05, 57M99, 06F15} 
 \keywords{Fundamental group, Dehn surgery, left-orderability, L-space conjecture}
\thanks{Adam Clay
    was partially supported by NSERC grant RGPIN-2020-05343.}

\date{\today}

\begin{document}

\begin{abstract}

    We introduce the conjugate-slope property of knots, a variant of Nie's Property (D), which can be used to study non-left-orderability of fundamental groups arising from Dehn filling.  We show that this property holds for any knot in \(S^3\) admitting a diagram whose negative crossings can be ``controlled'' in a precise combinatorial sense.  The conjugate-slope property also places strong restrictions on the kinds of left-orderings that the knot group can support, and implies that the image of the meridian is generalised torsion in the fundamental group of many of the knot's Dehn fillings.
\end{abstract}

\maketitle

\section{Introduction}

For a rational homology \(3\)-sphere, the L-space conjecture posits that having a left-orderable fundamental group, supporting a co-orientable taut foliation, and not being a Heegaard Floer homology L-space are equivalent \cite{BGW13, Juh15}.  Progress in the study of Heegaard Floer homology, co-orientable taut foliations, and left-orderings of fundamental groups has proceeded at different rates. As a result, certain \(3\)-manifolds and topological operations whose behaviours are well understood with respect to one of these objects may remain poorly understood with respect to the others.

Notably, Dehn surgery along a knot in \(S^3\) is well understood with respect to Heegaard Floer homology and L-spaces.  For a nontrivial knot \(K\) in \(S^3\), either all nontrivial Dehn fillings of the complement yield non-L-spaces, or the Dehn fillings which yield non-L-spaces are precisely the ones with filling slope less than \(2g(K)-1\), where \(g(K)\) is the Seifert genus \cite{OS05}.  Knots satisfying the latter property are called L-space knots.  In contrast, the understanding of left-orderability of fundamental groups of Dehn fillings is rather piecemeal, with differing approaches yielding only partial results \cite{BGW13, CW11, CW13, CD18, BH19}.

One of the challenges in tackling the non-left-orderability aspect of Dehn filling is that most results have been obtained through rather ad-hoc approaches.  A typical approach is to single out a family of L-space knots for study, and to perform computations particular to a specific presentation of their knot groups in order to obstruct left-orderability of the fundamental groups of the Dehn fillings.  It would be better if one's approach did not rely on the computational properties of a particular group presentation, but instead was immediately connected to the topological and geometric properties of the knot.

A step in this direction is Nie's Property (D), which potentially characterises L-space knots in terms of two properties that we will refer to as parts (a) and (b) (\cite{Nie20}, and also see Definition \ref{propertyD}).  The strength of this approach is that each property alone is \emph{a priori} weaker than the property of being an L-space knot, allowing for a two-step approach to the problem of group-theoretically characterising L-space knots.  The families of knots satisfying Property (D) part (a) and Property (D) part (b) could each properly contain the collection of all L-space knots, with the collection of L-space knots being their intersection. Property (D) also codifies an approach which, at least in principle, may be carried out without reference to a particular presentation of the knot group.

With this in mind, we present a modified version of Nie's Property (D), consisting of two parts that we call ``the conjugate-slope property at \(n\)'', where \(n \in \mathbb{N}\), and the ``cofinal detection property'' (see Definitions \ref{conjugate-slope} and \ref{def:cofinal}).  Our conjugate-slope property is very similar to but weaker than Nie's Property (D) part (b), and we conjecture that our conjugate-slope property at \(2g(K)-1\) is equivalent to Property (D) part (b).   The cofinal detection property is also weaker than Property (D) part (a).  We show that our two properties each define a family that is not equal to the collection of L-space knots, and we conjecture that the knots satisfying both are precisely the L-space knots.

Our aim is to connect these newly defined group-theoretic properties with topological and geometric properties of knots.  With this in mind, this manuscript primarily focuses on the conjugate-slope property and its amenability to investigation via knot diagrams.  We develop techniques for identifying families of knots that have this property, and study its impact on the types of orderings that can be supported by the knot group.  We show:

\begin{theorem}
    Suppose that \(K \subset S^3\) is a nontrivial knot that admits a diagram with two or fewer negative crossings.  Then \(K\) has the conjugate-slope property at some \(n \geq 2g(K) -1\).
\end{theorem}

More generally, any knot admitting a diagram whose negative crossings satisfy a technical combinatorial condition will have the conjugate-slope property at some \(n \geq 2g-1\). Knots with two or fewer negative crossings trivially satisfy the required condition.  See Theorem \ref{thm:con_slope_property} and Corollary \ref{cor:conj_slope_bound}.  In particular, all L-space knots in the SnapPy census have the conjugate-slope property, as do many infinite families of L-space knots.

The conjugate-slope property also restricts the kinds of orderings that can be supported by the knot group.  Of particular importance are the boundary-cofinal left-orderings, that is, left-orderings of the knot group relative to which the peripheral subgroup is unbounded \cite{BC23}.  Conjecturally, all orderings of L-space knot groups are of this kind \cite[Theorem 1.2]{BC23}, and every knot group supports uncountably many orderings of this type.  The conjugate-slope property places strong restrictions on the possible boundary behaviour of boundary-cofinal left-orderings.

\begin{theorem}
    If \(K \subset S^3\) has the conjugate-slope property at \(n\), then no boundary-cofinal ordering of the knot group weakly order-detects a slope greater than \(n\).
\end{theorem}

See \cite{BC23} and the background section below for the definition of weak order-detection, and Proposition \ref{prop:cofinal_not_allowed}.  Combined with the cofinal detection property, this is exactly the result needed to obstruct left-orderability of the fundamental groups arising from Dehn filling, see Theorem \ref{thm:nonLO}.

\subsection{Organisation of the paper}
In Section \ref{sec:background} we introduce left-orderings, slope detection, and other definitions essential to our analysis of orderings on knot complements.  Section \ref{sec:con_slope} introduces the conjugate-slope property and shows that many families of knots have the conjugate-slope property.  Section \ref{sec:weak_detection} introduces Property (D), \(r\)-decay, the cofinal detection property, and investigates weak order-detection of slopes and the relationships between these properties.  In  Section \ref{sec:consequences} we show how to obstruct left-orderability using the conjugate-slope property, and show how the conjugate-slope property can be used to produce generalised torsion in the fundamental groups of Dehn-filled manifolds.
\section{Background}
\label{sec:background}

A \emph{left-ordering} \(\ooo\) on a non-trivial group \(G\) is defined by a total order \(<_\ooo\) on \(G\) such that \(a<_\ooo b\) implies \(ga<_\ooo gb\) for all elements \(g,a,b\in G\). The \emph{positive cone} associated with a left-ordering \(\ooo\) is the set  \(P(\ooo)=\{g\in G\mid g>_\ooo id\}\), which partitions the group as \(G=P(\ooo)\sqcup P(\ooo)^{-1}\sqcup\{id\}\).  We call the elements of \(P(\ooo)\) (resp. \(P(\ooo)^{-1}\)) \emph{positive} (resp. \emph{negative}).  Conversely, given a semigroup \(P \subset G\) satisfying \(G=P\sqcup P^{-1}\sqcup\{id\}\), we can specify a left-ordering \(\ooo\) by taking \(a<_\ooo b\) if \(a^{-1}b\in P\) for \(a,b\in G\). Hence, every left-ordering is naturally identified with its positive cone.

Given a left-ordering \(\ooo\) on \(G\), a subset \(C\subset G\) is said to be \(\ooo\)-\emph{convex} if for any \(a, b \in C\) and \(g \in G\), the condition \(a<_\ooo g<_\ooo b\) implies \(g\in C\). Given a subset \(A\subset G\), the \(\ooo\)-convex hull of \(A\) is defined as \[C(A)=\{g\in G\mid a\leq_\ooo g\leq_\ooo b \mbox{ for some }a,b\in A\}.\] A subset \(A\) is \(\ooo\)-\emph{cofinal} if \(C(A)=G\), and an element \(g\in G\) is \(\ooo\)-\emph{cofinal} if \(\inner{g}\) is \(\ooo\)-cofinal.

Let \(\LO(G)\) denote the set of all left-orderings on \(G\). When equipped with the Sikora topology \cite{Sikora04}, \(\LO(G)\) turns out to be a compact, Hausdorff and totally disconnected space. There is also a canonical action of \(G\) on \(\LO(G)\) by homeomorphisms via setting \[P(g\cdot\ooo)=gP(\ooo)g^{-1}.\] Given a subset \(\mathcal{N}\subset \LO(G)\), we say \(\mathcal{N}\) is \emph{normal} if \(\mathcal{N}\) is \(G\)-invariant, meaning \(g\cdot \ooo\in\mathcal{N}\) for all \(g\in G\) and \(\ooo\in\mathcal{N}\).

An important example in the study of order-detection is the case \(G=\ZZ^2\), which we identify as the lattice points in \(H_1(\ZZ^2;\RR)=\ZZ^2\otimes \RR=\RR^2\). The structural properties of positive cones imply that any left-ordering \(\ooo\) on \(\ZZ^2\) determines a line \(L(\ooo)\subset \RR^2\) passing through the origin, such that elements strictly on one side of \(L(\ooo)\) are all positive and those on the other side are negative. Consequently, a left-ordering \(\ooo\) on \(\ZZ^2\) determines a slope-specifically, the slope of \(L(\ooo)\), denoted by \([L(\ooo)]\).

By a \emph{slope} on a torus \(T\), we mean an element in the projective space of \(H_1(T;\RR)\). We use \([\gamma]\) to denote the slope corresponding to a nonzero element \(\gamma\in H_1(T;\RR)\). Identifying \(H_1(T; \mathbb{Z})\) with the lattice points in \(H_1(T; \RR)\), we call \([\gamma]\) \emph{rational} if \(\gamma\in H_1(T; \mathbb{Z})\); in this case, we will require \(\gamma\) to be primitive by convention. Let \(\mathcal{S}(T)\) denote the set of slopes on \(T\), noting that \(\mathcal{S}(T) = P H_1(T; \RR) \cong S^1\). Since \(\pi_1(T) = H_1(T; \mathbb{Z}) = \ZZ^2\), a left-ordering on \(\pi_1(T)\) yields a slope corresponding to \([L(\ooo)]\). In fact, the \emph{slope map} \[s:\LO(\pi_1(T))\to \mathcal{S}(T),\; \ooo\mapsto [L(\ooo)]\] is continuous \cite[Chapter 6]{Clay2010}.

Now let \(K\) be a nontrivial knot in \(S^3\) with knot complement \(M\). The inclusion of the boundary torus \(\partial M\to M\) induces an embedding \(\pi_1(\partial M) \to \pi_1(M)\). Thus, a left-ordering \(\ooo\) on \(\pi_1(M)\) uniquely determines a slope \(s(\ooo|_{\pi_1(\partial M)})\) on the boundary torus \(\partial M\), where \(\ooo|_{\pi_1(\partial M)}\) is the restriction of the left-ordering \(\ooo\) to \(\pi_1(\partial M)\).

To simplify our discussion and notation, we often refer to slopes on \(M\) rather than \(\partial M\), and set \(\mathcal{S}(M) = \mathcal{S}(\partial M)\).  This gives us a well-defined slope map \[s: \LO(\pi_1(M))\to \mathcal{S}(M),\; \ooo\mapsto [L(\ooo|_{\pi_1(\partial M)})].\] Finally, a left-ordering \(\ooo \in \LO(\pi_1(M))\) is \emph{boundary-cofinal} if the peripheral subgroup \(\pi_1(\partial M)\) is \(\ooo\)-cofinal.

This manuscript largely concerns the notion of order-detection of slopes, which we are now ready to define.

\begin{definition}[\cite{BC23}]
    A slope \([\gamma]\in \mathcal{S}(M)\) is \emph{weakly order-detected} if there is a left-ordering \(\ooo\) on \(\pi_1(M)\) such that \(s(\ooo)=[\gamma]\). The slope \([\gamma]\) is \emph{order-detected} if there is a left-ordering \(\ooo\in \LO(\pi_1(M))\) such that \(s(g\cdot \ooo)=[\gamma]\) for all \(g\in \pi_1(M)\).
\end{definition}

For more information on order-detection of slopes, see \cite{BC23, CL25}.



\section{The conjugate-slope property}

\label{sec:con_slope}

Let \(G\) be a group. Given a subset \(S \subset G\) we use the notation \(N(S)\) to denote the normal subsemigroup of \(G\) generated by \(S\), and therefore
\[ N(S) = \left\{ \prod_{i=1}^m h_i g_i h_i^{-1} \mid h_i \in G, g_i \in S, m \in \mathbb{Z}^+ \right\}.
\]
We use \(R(S)\) to denote the set of positive roots of elements of \(S\), that is,
\[ R(S) = \{ g \in G \mid g^n \in S \mbox{ for some positive integer } n\}.
\]

Fix a nontrivial knot \(K\subset S^3\) and a choice of meridian and longitude \(\mu, \lambda \in G\), where \(G\) is the knot group of \(K\).

\begin{definition}
    \label{conjugate-slope}
    Let  \(n\) be a positive integer.  We say that a nontrivial knot \textit{\(K\) has the conjugate-slope property at \(n\)}, or simply \textit{\(K\) has the conjugate-slope property}, if \(N(\mu^n \lambda) \cap N(\mu) \neq \emptyset\).
\end{definition}

\begin{proposition}
    If \(K\) has the conjugate-slope property at \(n\), then \(K\) has the conjugate-slope property at every \(m>n\).
\end{proposition}
\begin{proof}
    Suppose that \(N(\mu^n \lambda) \cap N(\mu) \neq \emptyset\).  We show that \(K\) has the conjugate-slope property at \(n+1\), and the claim then follows by induction.

    We have
    \[ \prod_{j=1}^\ell h^{-1}_j \mu h_j= \prod_{i=1}^k g^{-1}_i\mu^n \lambda g_i
    \]
    and note that we can rewrite the right-hand side as
    \[ \prod_{i=1}^k (g_i^{-1}
        \mu^{-1} g_i)(g^{-1}_i\mu^{n+1} \lambda g_i).
    \]
    This expression becomes a product of conjugates of \(\mu^{n+1}\lambda\) when we left-multiply by \(\prod_{i=0}^{k-1}(g_{k-i}^{-1} \mu g_{k-i})\).  Our first equality becomes
    \[  \prod_{i=0}^{k-1}(g_{k-i}^{-1} \mu g_{k-i}) \prod_{j=1}^\ell h^{-1}_j \mu h_j = \prod_{i=0}^{k-1}(g_{k-i}^{-1} \mu g_{k-i})\prod_{i=1}^k (g_i^{-1}
        \mu^{-1} g_i)(g^{-1}_i\mu^{n+1} \lambda g_i).
    \]
    The left-hand side of this equality lies in \(N(\mu)\), while the right-hand side is an element of \(N(\mu^{n+1} \lambda)\).
\end{proof}
Consequently, when we begin investigating which knots have the conjugate-slope property, we will always aim to prove that the conjugate-slope property holds for the smallest possible positive integer \(n\).

To this end, let \(K\) be a nontrivial knot, and \(D\) an oriented diagram of \(K\).  Let \(A_1, \ldots, A_m\) denote the oriented arcs of the diagram, and denote the crossings by \(C_1, \ldots, C_n\).  Corresponding to each crossing is a collection of three arcs, two of them being the under-arcs of the crossing and one of them being the over-arc. Each crossing is either positive or negative according to the right-hand rule, and the crossings themselves are cyclically ordered according to the order that we \emph{pass under each over-arc} as we travel along the oriented arcs of the diagram.  Without loss of generality, we assume that the crossings are indexed so that \((C_1, \ldots, C_n)\) is, up to cyclic permutation, the natural cyclic ordering of the crossings.  We say that an arc \(A_i\) is \emph{between} crossings \(C_k\) and \(C_\ell\) if, after passing under the crossing \(C_k\), we travel along the arc \(A_i\) before passing under the crossing \(C_\ell\).

We say that the \emph{negative crossings of \(D\) occur in two batches} if there exist:
\begin{enumerate}
    \item Two arcs \(A_i, A_j \in \{A_1, \ldots, A_m\}\) such that for every negative crossing in the diagram, either \(A_i\) or \(A_j\) is the over-arc of the crossing.
    \item Integers \(r, s \in \{ 1, \ldots, n\}\) such that every negative crossing with over-arc \(A_i\) lies in \(\{ C_r, C_{r+1}, \ldots, C_{s-1}\}\) and every negative crossing with over-arc \(A_j\) lies in \(\{C_s, C_{s+1}, \ldots, C_{r-1}\}\), where we interpret the subscripts of the crossings in each set modulo \(n\).
    \item The arc \(A_i\) is between two crossings in \(\{ C_{r-1}, C_r, \ldots, C_{s}\}\) and the arc \(A_j\) is between two crossings in \(\{C_{s-1}, C_s, \ldots, C_{r}\}\).
\end{enumerate}
If there are \(k\) negative crossings with over-arc \(A_i\) and \(\ell\) negative crossings with over-arc \(A_j\), we say that \(k\) and \(\ell\) are the \emph{size} of each batch of negative crossings.  Either batch may be empty; in particular, positive knots satisfy the definition vacuously.

Nontrivial examples of knots having two batches of negative crossings can be constructed using (for example) braid closures, but in particular we would like to highlight the following.

\begin{proposition}
    \label{prop:knotsum}
    Suppose that \(K_1\) and \(K_2\) are nontrivial knots that each admit a diagram whose negative crossings have a common over-arc.   Then the knot sum \(K_1 \# K_2\) admits a diagram whose negative crossings occur in two batches.
\end{proposition}
\begin{proof}
    Suppose that \(K_1\) admits an oriented diagram with crossings \(C_1, \ldots, C_j\) (appearing in that order), and that \(K_2\) admits an oriented diagram with crossings \(C'_1, \ldots, C'_k\) (again, appearing in the given order).  Fix an arc \(A_i\) in the diagram of \(K_i\) for \(i=1,2\) such that all negative crossings of the diagram share \(A_i\) as a common over-arc.

    The connected sum \(K_1 \# K_2\) is the nontrivial knot constructed by deleting a short segment from an arc in \(K_1\) and a short segment from an arc in \(K_2\), and connecting the resulting endpoints.  This construction is independent of the choice of deleted segments, and so in particular one may choose segments in such a way that the arcs \(A_1\) and \(A_2\) are untouched.
    Then \(K_1 \# K_2\)  admits an oriented diagram with crossings \(C_1, \ldots, C_j, C'_1, \ldots, C'_k\), with their cyclic order as written. Then every negative crossing with over-arc \(A_1\) lies in \(\{C_1, \ldots, C_j\}\) and every negative crossing with over-arc \(A_2\) lies in \(\{C'_1, \ldots, C'_k\}\).  Moreover, \(A_1\) is between two crossings in \(\{C'_k, C_1, \ldots, C_j, C'_1\}\) since \(A_1\) is an arc from the diagram for \(K_1\), and \(A_2\) is between two crossings in \(\{C_j, C'_1, \ldots, C'_k, C_1\}\) since \(A_2\) is an arc from the diagram for \(K_2\).
\end{proof}

\begin{theorem}
    \label{thm:con_slope_property}
    Let \(K\) be a nontrivial knot in \(S^3\) with knot group \(G\).  Suppose that \(K\) admits an oriented diagram with \(v\) positive crossings and whose negative crossings occur in two batches of size \(k\) and \(\ell\).  Set \(N=v+|k-\ell|\).  Then there exists a choice of meridian and longitude \(\mu, \lambda\) and elements \(g, h_1, \ldots, h_{2N} \in G\) such that
    \[ \mu^N \lambda g^{-1}(\mu^N \lambda)g = \prod_{i=1}^{2N} h_i^{-1}\mu h_i.
    \]
    In particular, if \(K\) is a knot whose negative crossings occur in two batches of size \(k\) and \(\ell\), then \(K\) has the conjugate-slope property at \(N\).
\end{theorem}
\begin{proof}
    Fix a diagram \(D\) of \(K\) with two batches of crossings as in the statement of the theorem.  Recall that for each arc \(A_i\) in the diagram, there is a corresponding Wirtinger generator \(x_i\).  Suppose that the arcs \(A_i\) and \(A_j\) in the diagram \(D\) have \(k\) and \(\ell\) negative under-crossings as in the definition of two negative batches.  We choose the Wirtinger generator \(x_i\) as our meridian \(\mu\).

    We can write down an expression for the longitude \(\lambda\) of \(K\) as follows \cite[Remark 3.14]{BurdeGerhard2013K}.  Begin at a point on the arc \(A_i\).  Proceed along the arcs of the diagram following the orientation of \(D\), and write the Wirtinger generator \(x_k^{\epsilon_k}\) when crossing under the arc \(A_k\), with \(\epsilon_k \in \{ \pm 1\}\) recording the sign of the crossing.  Upon returning to our starting point on \(A_i\), we write \(\mu^{-\omega(D)}\) where \(\omega(D)\) is the writhe of the diagram.

    Using the fact that our negative crossings come in two batches, we arrive at
    \[ \lambda = P_0 \mu^{-1} P_1 \mu^{-1} \cdots \mu^{-1} P_tx_j^{-1} Q_1 x_j^{-1} Q_2 \cdots x_j^{-1}Q_{\ell} \mu^{-1} P_{t+1} \mu^{-1} \cdots \mu^{-1}P_k \mu^{-\omega(D)}
    \]
    where \(P_i, Q_j\) are products of Wirtinger generators with positive exponents, and \(0 \leq t \leq k\).  Note that \(\omega(D) = v-k-\ell\), and so we have
    \[\lambda \mu^{v-k-\ell}= P_0 \mu^{-1} P_1 \mu^{-1} \cdots \mu^{-1} P_t x_j^{-1} Q_1 x_j^{-1} Q_2 \cdots x_j^{-1}Q_{\ell} \mu^{-1} P_{t+1} \mu^{-1} \cdots \mu^{-1}P_k.
    \]
    When \(t=0\) (resp. \(t=k\)) we interpret the prefix (resp. suffix) of the product to consist only of \(P_0\) with no occurrences of \(\mu^{-1}\) (resp. consist only of \(P_k\) with no occurrences of \(\mu^{-1}\)).  Note that if there are no negative crossings (meaning \(k=\ell = 0\)), then we can already reach the desired conclusion since the right-hand side consists of a product of Wirtinger generators, all of which are conjugate to \(\mu\). Otherwise, we proceed as follows.

    Multiplying on the left and right by \(\mu^{t}\) and \(\mu^{k-t}\), we arrive at
    \[ \lambda \mu^{v-\ell} = \mu^tP_0 \mu^{-1} P_1 \mu^{-1} \cdots \mu^{-1} P_tx_j^{-1} Q_1 x_j^{-1} Q_2 \cdots x_j^{-1}Q_{\ell} \mu^{-1} P_{t+1} \mu^{-1} \cdots \mu^{-1}P_k \mu^{k-t}.
    \]
    Note that every Wirtinger generator is conjugate to \(\mu\), and therefore
    \[
        P = \mu^tP_0 \mu^{-1} P_1 \mu^{-1} \cdots \mu^{-1} P_t  \mbox{ and } P' = \mu^{-1} P_{t+1} \mu^{-1} \cdots \mu^{-1}P_k \mu^{k-t}
    \]
    can both be expressed as products of conjugates of \(\mu\), since each \(P_i\) is already a product of conjugates of \(\mu\).  Therefore, rewriting the equation above we have
    \[ \lambda \mu^{v-\ell} = Px_j^{-1} Q_1 x_j^{-1} Q_2 \cdots x_j^{-1}Q_{\ell} P'
    \]
    which upon right-multiplying by \(x_j^\ell\) becomes
    \[
        \lambda \mu^{v-\ell}x_j^\ell = Px_j^{-1} Q_1 x_j^{-1} Q_2 \cdots x_j^{-1}Q_{\ell} P' x_j^\ell,
    \]
    where the right-hand side can clearly be expressed as a product of conjugates of \(\mu\), since each \(Q_j\) is already such, and so \(\lambda \mu^{v-\ell}x_j^\ell \in N(\mu)\).  In particular, note that if \(\ell = 0\) then there are no occurrences of \(x_j^{-1}\) on the right-hand side, and the proof can be concluded easily from here.  Otherwise, we must use the second batch of negative crossings as follows.

    We repeat the same procedure as above, but we begin reading the formula for the longitude by starting on the arc \(A_j\).  Let \(g \in G\) denote the element satisfying \(g^{-1} \mu g = x_j\).  Recording the Wirtinger generators at each under-crossing as before, but starting on the arc \(A_j\), we have
    \[ g^{-1} \lambda g  = R_0 x_j^{-1} \cdots x_j^{-1} R_s \mu^{-1} S_1 \cdots \mu^{-1}S_k x_j^{-1} R_{s+1} \cdots x_j^{-1}R_\ell x_j^{-\omega(D)},
    \]
    where \(R_i\), \(S_j\) are products of Wirtinger generators, and we interpret the cases \(s=0\) and \(s=\ell\) similar to the previous case. As before, we may left and right multiply by appropriate powers of \(x_j\) and \(\mu\) in order to write \(g^{-1} \lambda \mu^{v-k} g \mu^k\) as a product of conjugates of \(\mu\) and conclude \(g^{-1} \lambda \mu^{v-k} g \mu^k \in N(\mu)\).

    Overall, we may therefore conclude that
    \[ \lambda \mu^{v-\ell}x_j^\ell \cdot g^{-1} \lambda \mu^{v-k} g \mu^k  \in N(\mu).
    \]
    Using \(g^{-1} \mu g = x_j\) and conjugating by \(\mu^k\) we conclude \[\lambda \mu^{v-\ell+k} g^{-1} \lambda \mu^{v-k+\ell} g \in N(\mu).\]

    If \(k>\ell\) then we right multiply by \(g^{-1}\mu^{2(k-\ell)}g\) and conclude
    \(\lambda \mu^{v-\ell+k} g^{-1} \lambda \mu^{v-k+\ell} g \cdot g^{-1}\mu^{2(k-\ell)}g = \lambda \mu^{v+|k-\ell|} g^{-1} \lambda \mu^{v+|k-\ell|} g \in N(\mu).\)
    If \(\ell >k\), we left-multiply by \(\mu^{2(\ell-k)}\) and similarly arrive at \[\lambda \mu^{v+|k-\ell|} g^{-1} \lambda \mu^{v+|k-\ell|} g \in N(\mu),\] whereas if \(k=\ell\) we can reach the same conclusion without any additional multiplication.

    In any event,  we have \[\lambda \mu^{v+|k-\ell|} g^{-1} \lambda \mu^{v+|k-\ell|} g \in N(\mu)\] which yields
    \[\mu^{v+|k-\ell|} \lambda g^{-1}(\mu^{v+|k-\ell|} \lambda)g = \prod_{i=1}^{2(v+|k-\ell|)} h_i^{-1}\mu h_i
    \]
    for appropriate choices of \(h_i\).
\end{proof}

\begin{example}
    In \cite{Himeno25}, the author produces a family \(K_n\) of hyperbolic L-space knots, where \(n \geq 2\). They are not braid positive when \(n\) is an even positive integer.

    The knot \(K_n\) has genus
    \[ g(K_n) = \frac{3}{2}n^2 -\frac{1}{2}n +1,
    \]
    and admits a diagram having \(3n^2 + 2n\) positive crossings, and \(n-1\) negative crossings (see \cite[Lemma 2.11]{Himeno25} and \cite[Figure 2]{Himeno25}).  Moreover, the \(n-1\) crossings all share a common over-arc.  As such, Theorem \ref{thm:con_slope_property} applies, and we conclude that the knots \(K_n\) have the conjugate-slope property at \(3n^2 + 3n = 2g(K_n) +(4n-1)\).
\end{example}

\begin{corollary}
    If \(K\) is a nontrivial knot that admits a diagram with two or fewer negative crossings, then \(K\) has the conjugate-slope property.
\end{corollary}

As noted in the course of the proof of Theorem \ref{thm:con_slope_property}, the proof simplifies substantially when at least one of the batches of negative crossings is empty.  Owing to this simpler situation, we can improve our computations to show that a positive knot diagram with \(C\) crossings has the conjugate-slope property for some \(N<C\).  Below we give examples of such improvements, and their applications.

Recall that an arc in a diagram \(D\) of a knot \(K\) is called a \emph{bridge} if it is the over-arc of at least one crossing; the number of such crossings is called the \emph{length} of the bridge. Given a diagram \(D\), define the \emph{bridge length} of \(D\) to be the length of the longest bridge.

\begin{proposition}
    \label{prop:one_bridge}
    Suppose that \(K\) is a nontrivial knot admitting a positive diagram \(D\) with \(c\) crossings and bridge length \(d\), and let \(G\) denote the knot group of \(K\).  Then there exists a choice of meridian and longitude \(\mu, \lambda\) and elements \(h_1, \ldots, h_{c-d} \in G\) such that  \[\mu^{c-d}\lambda  = \prod_{i=1}^{c-d}h_i^{-1}\mu h_i.\]
    In particular, \(K\) has the conjugate-slope property at \(c-d\).
\end{proposition}
\begin{proof}
    Fix an oriented positive diagram \(D\) of \(K\) with \(c\) crossings, and having a bridge of length \(d\) with associated over-arc \(A_i\) and Wirtinger generator \(x_i\).  Choosing \(x_i\) as our meridian \(\mu\), we may begin at a point on the arc \(A_i\) and follow the orientation of the diagram to read off:
    \[ \lambda = P_0 \mu P_1 \mu \cdots \mu P_d \mu^{-c}.
    \]
    Here, each \(P_i\) is a product of Wirtinger generators, and thus \(P_i \in N(\mu)\) whenever \(P_i\) is nontrivial.  Rearranging, we have
    \[ \mu^{c-d}\lambda = \mu^{-d}P_0 \mu P_1 \mu \cdots \mu P_d \in N(\mu),
    \]
    from which the Proposition follows.
\end{proof}

\begin{proposition}
    \label{prop:two_bridges}
    Suppose that \(K\) is a nontrivial knot admitting a positive diagram \(D\) with \(c\) crossings and two bridges of length \(d\), and let \(G\) denote the knot group of \(K\). Then there exists a choice of meridian and longitude \(\mu, \lambda\) and elements \(g, h_1, \ldots, h_{c-d} \in G\) such that  \[g^{-1}(\mu^{c-d-1}\lambda)g\mu^{c-d-1}\lambda  = \prod_{i=1}^{2(c-d-1)}h_i^{-1}\mu h_i.\]
    In particular, \(K\) has the conjugate-slope property at \(c-d-1\).
\end{proposition}
\begin{proof}
    Fix an oriented positive diagram \(D\) of \(K\) with \(c\) crossings, and having two bridges of length \(d\),  with associated over-arcs \(A_i\) and \(A_j\), with corresponding Wirtinger generators \(x_i\) and \(x_j\) respectively.  We will fix \(x_i\) as our choice of meridian \(\mu\).

    Beginning at a point on the arc \(A_i\), we read off \(\lambda = P_0 \mu P_1 \mu \cdots \mu P_d \mu^{-c}\), as in the proof of Proposition \ref{prop:one_bridge}.  Choose \(P_i\) such that \(P_i = R x_j R'\), where \(R, R'\) are products of Wirtinger generators.  Then we have
    \[ \lambda = P_0 \mu P_1 \mu \cdots \mu R x_j R' \mu \cdots  \mu P_d \mu^{-c}
    \]
    which yields
    \[
        x_j^{-1}\mu^{c-d} \lambda = x_j^{-1}\mu^{-i}P_0 \mu P_1 \mu \cdots \mu R x_j R' \mu \cdots  \mu P_d \mu^{-(d-i)},
    \]
    note that the right-hand side is an element of \(N(\mu)\).

    Fix \(g \in G\) satisfying \(g^{-1} \mu g = x_j\).  Then, beginning at a point on the arc \(A_j\) and following the orientation of the diagram, we read off \( g^{-1} \lambda g = Q_0 x_j Q_1 x_j \cdots x_j Q_d x_j^{-c}\), and rearranging similar to above, we find \(\mu^{-1} g^{-1}\mu^{c-d}\lambda g \in N(\mu)\).  This means that the product
    \[ (\mu^{-1} g^{-1}\mu^{c-d}\lambda g)(x_j^{-1}\mu^{c-d} \lambda) = \mu^{-1} (g^{-1} \mu^{c-d-1}\lambda g ) \mu^{c-d} \lambda
    \]
    lies in \(N(\mu)\). Conjugating the expression on the right by \(\mu^{-1}\) we arrive at \[(g^{-1} \mu^{c-d-1}\lambda g ) \mu^{c-d-1} \lambda \in N(\mu),\]
    from which the proposition follows.
\end{proof}

Recall that the braid \(\delta_n \in B_n\) is given by
\(\delta_n = \sigma_1 \cdots \sigma_{n-1}.\)

\begin{figure}
    \centering
    \includegraphics[width=0.28\linewidth]{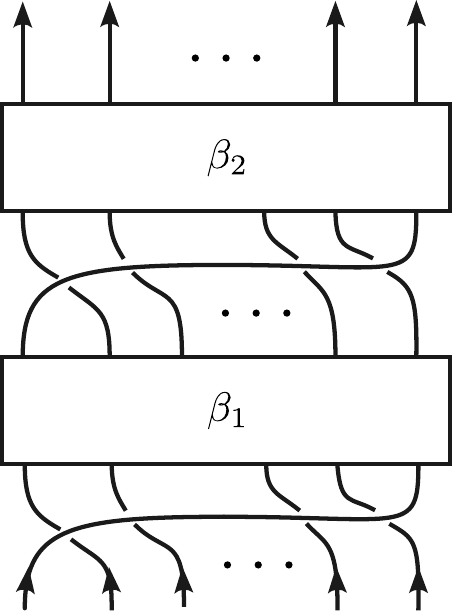}
    \caption{The braid \(\delta_n \beta_1 \delta_n \beta_2\).}
    \label{fig:deltan}
\end{figure}

\begin{corollary}
    \label{cor:cspropertyat2g-1}
    If \(K\) is the closure of an \(n\)-braid of the form \(\delta_n \beta_1 \delta_n \beta_2\), where \(\beta_1, \beta_2 \in B_n\) are positive braids, then \(K\) has the conjugate-slope property at \(2g(K)-1\).\footnote{Nie has recently informed the authors via private communication that he can prove the same conclusion holds for every braid-positive knot.}
\end{corollary}
\begin{proof}
    Suppose that \(K\) is the closure of a positive \(n\)-braid of the form \(\delta_n \beta_1 \delta_n \beta_2\) (see Figure \ref{fig:deltan}), having \(c\) crossings in total. Seifert's algorithm readily produces a minimal genus Seifert surface having \(n\) Seifert circles, from which we calculate that \(g(K) = \frac{1}{2}(c-n+1)\).

    At the same time, since there is a diagram of \(K\) having two bridges of length \(n-1\) with each bridge arising from an instance of \(\delta_n\), Proposition \ref{prop:two_bridges} says that \(K\) has the conjugate-slope property at \(c-n = 2g(K)-1\).
\end{proof}

Many knots therefore have the conjugate-slope property  at \(2g(K)-1\), for example twist positive knots (positive braid knots where the braid word contains a positive full twist), which encompasses torus knots, \(1\)-bridge braids, algebraic knots, and Lorenz knots \cite{KM25}.  In fact, most L-space knots are twist positive, and so have the conjugate-slope property at \(2g(K)-1\) \cite[Section 3]{KM25}.

We also observe that from Theorem \ref{thm:con_slope_property} and \cite{BK24}, if \(K\) is a census L-space knot then \(K\) has the conjugate-slope property, because all census L-space knots are braid positive, with one exception.  The knot \(o9\)\_\(30634\) is an L-space knot which is not braid positive, but it is the closure of a \(3\)-strand braid with exactly one negative crossing.  Therefore it has the conjugate-slope property by Theorem \ref{thm:con_slope_property}.  With this evidence in hand and from connections to the L-space conjecture discussed in the next section (in particular, Corollary \ref{cor:conj_slope_bound}), we conjecture:

\begin{conjecture}
    Let \(K\) be a nontrivial L-space knot.  Then \(K\) has the conjugate-slope property at \(2g(K)-1\).
\end{conjecture}

More generally, as many non-L-space knots have the conjugate-slope property, we ask:

\begin{question}
    If \(K\) has the conjugate-slope property, does \(K\) have the conjugate-slope property at \(2g(K) -1\)?
\end{question}

In light of this question, let us briefly revisit Proposition \ref{prop:one_bridge}, using the notation found there. Let \(Q_K(z)\) denote the \(Q\)-polynomial of a knot \(K\) \cite{BLM86, Ho86}, equivalently it is the Kauffman polynomial \(F_K(a,z)\) evaluated at \(a=1\).  It is a result of Kidwell that
\[\deg_z(Q_K) \leq c-d,
\]
no matter the diagram \(D\) used to produce the values \(c\) and \(d\) appearing on the right-hand side \cite{Kidwell87}.  Combined with Corollary \ref{cor:conj_slope_bound} below, our techniques show the following.

\begin{corollary}
    \label{cor:c-dbound}
    For a diagram \(D\) of a nontrivial knot \(K\) in \(S^3\), let \(c(D)\) denote the number of crossings appearing in the diagram, and \(d(D)\) the bridge length of \(D\).  Then
    \[ 2g(K)-1 \leq c(D)-d(D)
    \]
    for every positive diagram \(D\) of \(K\).
\end{corollary}

We suspect it would be interesting to investigate whether there are knots admitting a positive diagram for which the inequality in Corollary \ref{cor:c-dbound} becomes an equality, and to better understand the role of Kidwell's \(Q\)-polynomial bound in this setting.

\section{The conjugate-slope property, Property (D), and weak order-detection of slopes}

\label{sec:weak_detection}

Our inspiration for investigating the conjugate-slope property lies in its connections with the L-space conjecture.  In this section we investigate the relationship between the conjugate-slope property, and other properties defined in the literature for the purpose of studying non-left orderability of fundamental groups arising from Dehn filling.

Use \(T(S)\) to denote the smallest root-closed, conjugacy-closed subsemigroup containing a set \(S\).

\begin{lemma}
    \label{TS_structure}
    Let \(G\) be a group and \(S \subset G\) any subset.  Set \(S_0 = N(S)\) and, for all \(i \geq 1\), set \(S_i = N(R(S_{i-1}))\).  Then \(T(S) = \bigcup_{i=0}^{\infty} S_i\).
\end{lemma}
\begin{proof}
    Note that \(\bigcup_{i=0}^{\infty} S_i\) is a subsemigroup, since \(S_0 \subset S_1 \subset S_2 \subset \ldots\) and each \(S_i\) is a subsemigroup.  Moreover, \(\bigcup_{i=0}^{\infty} S_i\) is root-closed.  For if \(g \in G\) and \(g^n \in \bigcup_{i=0}^{\infty} S_i\) for some \(n>0\), then \(g^n \in S_k\) for some \(k \geq 0\), and so \(g \in R(S_k) \subset S_{k+1}\).  It follows that \(T(S) \subset \bigcup_{i=0}^{\infty} S_i\).

    On the other hand, since every element in \(\bigcup_{i=0}^{\infty} S_i\) is a result of successively taking products of conjugates of roots of elements of \(N(S)\), we conclude that \(\bigcup_{i=0}^{\infty} S_i \subset T(S)\).
\end{proof}

For the definitions and proofs that follow in this section, we fix a nontrivial knot \(K\) in \(S^3\) with complement \(M\), and let \(\{ \mu, \lambda \} \subset \pi_1(M)\) be a choice of meridian and longitude for \(K\).

\begin{definition}[\cite{Nie20}]
    \label{propertyD}
    Let \(g(K)\) be the Seifert genus of \(K\).  Then \(K\) \emph{has property (D)} if
    \begin{enumerate}[label=(\alph*)]
        \item given a homomorphism \(\rho : \pi_1(M) \rightarrow \mathrm{Homeo}_+(\mathbb{R})\), if \(t \in \mathbb{R}\) is fixed by both \(\rho(\mu)\) and \(\rho(\lambda)\) then \(t\) is fixed by \(\rho(g)\) for every \(g \in \pi_1(M)\); and
        \item \(\mu \in T(\mu^{2g(K)-1}\lambda, \mu^{-1})\).
    \end{enumerate}
\end{definition}

We begin with a study of Nie's Property (D) part (b) and the conjugate-slope property.  In the proofs below we use the notation \(g^h = h^{-1}gh\) as a shorthand for conjugation of group elements.

\begin{lemma}
    \label{lem:semigroup_equivalence}
    Suppose \(k \neq 0\) is an integer.
    With the same setup as above, \(N(\mu)\cap N(\mu^k\lambda)\not=\emptyset\) if and only if \(id\in N(\mu^k\lambda,\mu^{-1})\).
\end{lemma}

\begin{proof}
    Assume first that \(x\in N(\mu)\cap N(\mu^k\lambda)\). Since \(x\in N(\mu)\), we may write \(x=\mu^{g_1}\cdots \mu^{g_m}\) for some \(g_1,\dots,g_m\in \pi_1(M)\). Similarly, \(x= (\mu^k\lambda)^{h_1}\dots(\mu^k\lambda)^{h_n}\) for some \(h_1,\dots,h_n\in \pi_1(M)\) since \(x \in N( \mu^k \lambda)\). Therefore, \[id= (\mu^k\lambda)^{h_1}\dots(\mu^k\lambda)^{h_n}(\mu^{g_1}\dots \mu^{g_m})^{-1}=(\mu^k\lambda)^{h_1}\dots(\mu^k\lambda)^{h_n}(\mu^{-1})^{g_m}\dots(\mu^{-1})^{g_1},\] from which it follows that \(id\in N(\mu^k\lambda,\mu^{-1})\).

    For the other direction, assume \(id\in N(\mu^k\lambda,\mu^{-1})\), and so we may write  \[id= h_1^{g_1}h_2^{g_2}\dots h_n^{g_n},\] where \(h_i\) is either \(\mu^k\lambda\) or \(\mu^{-1}\) and \(g_i\in \pi_1(M)\) for \(i=1,2,\dots,n\), where \(n>0\).

    Suppose first that \(h_j = \mu^k\lambda\) for all \(j = 1, \ldots, n\), and let \(\phi : \pi_1(M) \rightarrow \mathbb{Z}\) denote the abelianization map.  Then \(id= h_1^{g_1}h_2^{g_2}\cdots h_n^{g_n}\) implies \(0 = \phi(h_1^{g_1}h_2^{g_2}\cdots h_n^{g_n}) = nk\), a contradiction.

    So, there exists \(j\) such that \(h_j = \mu^{-1}\).  Assume \(h_{i_1}\) is the first occurrence of \(\mu^{-1}\) among \(h_1,h_2,\dots,h_n\). Set \(x_1=(\mu^k\lambda)^{g_1}(\mu^k\lambda)^{g_2}\cdots  (\mu^k\lambda)^{g_{i_1-1}}\), and we compute \begin{align*}
        id             & = x_1 (\mu^{-1})^{g_{i_1}}h_{i_1+1}^{g_{i_1+1}}\cdots h_n^{g_n}   \\
        id             & = (\mu^{-1})^{g_{i_1}'}h_{i_1+1}^{g_{i_1+1}'}\cdots h_n^{g_n'}x_1 \\
        \mu^{g_{i_1}'} & =h_{i_1+1}^{g_{i_1+1}'}\cdots h_n^{g_n'}x_1,
    \end{align*}
    where \(g_j'=g_jx_1^{-1}\) for \(j=i_1,i_1+1,\dots,n\). If \(h_j=\mu^{k}\lambda\) for \(j=i_1+1, \ldots, n\), then the previous equation gives \(N(\mu)\cap N(\mu^k\lambda)\not=\emptyset\) and we are done.  If not, assume \(h_{i_2}\) is the first occurrence of \(\mu^{-1}\) among \(h_{i_1+1},\dots,h_n\), and set \(x_2=h_{i_1+1}^{g_{i_1+1}'}\cdots h_{i_2-1}^{g_{i_2-1}'}\), and compute as before \begin{align*}
        \mu^{g_{i_1}'}                & =x_2 (\mu^{-1})^{g_{i_2}'} h_{i_2+1}^{g_{i_2+1}'}\cdots h_n^{g_n'}x_1,     \\
        \mu^{g_{i_1}'}                & = (\mu^{-1})^{g_{i_2}''} h_{i_2+1}^{g_{i_2+1}''} \cdots  h_n^{g_n''}x_2x_1 \\
        \mu^{g_{i_2}''}\mu^{g_{i_1}'} & =  h_{i_2+1}^{g_{i_2+1}''} \cdots  h_n^{g_n''}x_2x_1,
    \end{align*}
    where \(g_j''=g_j'x_2^{-1}\) for \(j=i_2,i_2+1,\dots,n\). If \(h_j=\mu^{k}\lambda\) for \(j=i_2+1, \ldots, n\) we are done; if not, we continue the rewriting procedure. This rewriting procedure stops after at most \(n-1\) steps, and from the final equation we conclude \(N(\mu)\cap N(\mu^k\lambda)\not=\emptyset\).
\end{proof}

\begin{lemma}
    \label{lem:root_closed}
    Suppose \(k >0 \) is an integer.
    With the same setup as above, \(\mu\in T(\mu^k\lambda,\mu^{-1})\) if and only if \(id \in T(\mu^k\lambda,\mu^{-1})\).
\end{lemma}
\begin{proof}
    One direction is clear: assuming \(\mu\in T(\mu^k\lambda,\mu^{-1})\), then \(id=\mu\mu^{-1}\in T(\mu^k\lambda,\mu^{-1})\).

    For the other direction, let \(S=\{\mu^k\lambda,\mu^{-1}\}\).  Then using the notation of Lemma \ref{TS_structure}, we begin with a preliminary claim.

    \noindent \textbf{Claim:} For each \(x\in S_t\setminus T(\mu^k\lambda)\), there exists \(g \in G\), \(r\in R(S_{t-1}) \setminus T(\mu^k\lambda)\), and \(s \in T(\mu^k \lambda, \mu^{-1})\) such that \(x^g = rs\).


    \noindent \textbf{Proof of claim:} Let \(x\in S_t\setminus T(\mu^k\lambda)\). Then \(x= s_1^{h_1}\cdots s_n^{h_n}\) for some \(h_1,\dots,h_n\in G\) and \(s_1,\dots,s_n\in R(S_{t-1})\). Since \(x \notin T(\mu^k \lambda)\) there exists a smallest index \(i_1 \in \{1,2,\dots,n\}\) such that \(s_{i_1}\in R(S_{t-1})\setminus T(\mu^k\lambda)\). Write \(x_1= s_1^{h_1} s_2^{h_2}\cdots s_{i_1-1}^{h_{i_1-1}}\). Then we can rewrite the expression for \(x\) as follows:
    \begin{align*}
        x & = x_1 s_{i_1}^{h_{i_1}}s_{i_1+1}^{h_{i_1+1}}\cdots s_n^{h_n}                            \\
        x & = s_{i_1}^{ h_{i_1}x_1^{-1}}s_{i_1+1}^{ h_{i_1+1}x_1^{-1}}\cdots s_n^{ h_nx_1^{-1}} x_1
    \end{align*}
    The second expression for \(x\) can be conjugated to yield
    \[
        x^{( h_{i_1}x_1^{-1})^{-1}}  =  s_{i_1} (s_{i_1+1}^{ h_{i_1+1}x_1^{-1}}\cdots s_n^{ h_nx_1^{-1}} x_1 )^{(h_{i_1}x_1^{-1})^{-1}}.
    \]

    Note that \((s_{i_1+1}^{ h_{i_1+1}x_1^{-1}}\cdots s_n^{ h_nx_1^{-1}} x_1 )^{(h_{i_1}x_1^{-1})^{-1}}\in S_t\). Thus, taking \(r = s_{i_1}\), \(s = (s_{i_1+1}^{ h_{i_1+1}x_1^{-1}}\cdots s_n^{ h_nx_1^{-1}} x_1 )^{(h_{i_1}x_1^{-1})^{-1}}\) and \(g = ( h_{i_1}x_1^{-1})^{-1}\)  proves the claim. \qed



    Armed with this claim, we proceed to the proof of the lemma.  Assume \(id \in T(\mu^k\lambda,\mu^{-1})\).  We first confirm that \(id \notin T(\mu^k \lambda)\).  To see this, let \(\phi: \pi_1(M) \rightarrow \mathbb{Z}\) denote the abelianization map, and set \(Q = \phi^{-1}(\mathbb{Z}_{>0})\).  One easily checks that \(Q\) is a root-closed, conjugacy-closed subsemigroup and that \(\mu^k \lambda \in Q\).  Therefore \(T(\mu^k \lambda) \subset Q\), and since \(id \notin Q\), this implies \(id \notin T(\mu^k \lambda)\).

    Thus there exists \(t \geq 0\) such that \(id\in S_t\setminus T(\mu^k\lambda)\).  We iteratively apply our claim as follows.

    We first apply our claim to \(id\),  yielding \(r_1 \in R(S_{t-1}) \setminus T(\mu^k \lambda)\) and \(s_1 \in T(\mu^k \lambda, \mu^{-1})\) such that \(id = r_1 s_1\).  Let \(k_1\) be the least positive integer such that \(r_1^{k_1} \in S_{t-1}\).

    Now for steps \(i=2, \ldots, t\) we apply our claim to \((r_{i-1})^{k_{i-1}} \in S_{t-{i-1}} \setminus T(\mu^k \lambda)\) and find \(g_i \in \pi_1(M)\), \(r_i \in R(S_{t-i})\) and \(s_i \in T(\mu^k \lambda, \mu^{-1})\) such that \(((r_{i-1})^{k_{i-1}})^{g_i} = r_i s_i\), and we let \(k_i\) denote the least positive integer such that \(r_i^{k_i} \in S_{t-i}\).

    Using these equations, we may substitute as follows.  The equation \(id = r_1 s_1\) implies that \(id = (r_1^{k_1})^{g_2} (s_1^{k_1})^{g_2}\), so that we may use the equation \((r_1^{k_1})^{g_2} = r_2 s_2\) to arrive at
    \[id = r_2s_2 (s_1^{k_1})^{g_2}.
    \]
    From this equation, we know that
    \[ id = (r_2^{k_2})^{g_3} ((s_2 (s_1^{k_1})^{g_2})^{k_2})^{g_3},
    \]
    and so we may substitute using \((r_2^{k_2})^g_3 = r_3 s_3\) to arrive at
    \[id = r_3s_3 ((s_2 (s_1^{k_1})^{g_2})^{k_2})^{g_3}
        .
    \]
    Continuing this pattern of substitutions, we arrive at
    \[ id = r_ts_t(s_{t-1}( \cdots (s_2 (s_1^{k_1})^{g_2})^{k_2} \cdots )^{k_{t-1}})^{g_t},
    \]
    and note that this equation is of the form \(id = r_t s\) where \(r_t \in S_0 \setminus T(\mu^k \lambda)\) and \(s \in T(\mu^k \lambda, \mu^{-1})\).  Recalling that \(S_0 = N(\mu^k \lambda, \mu^{-1})\), we may therefore write \[r_t^{k_t}= h_1^{f_1}h_2^{f_2}\cdots h_m^{f_m},\] where \(h_i \in \{\mu^k\lambda, \mu^{-1}\}\), \(f_i\in \pi_1(M)\) for \(i=1,2,\dots,m\).  Since \(r_t \notin T(\mu^k \lambda)\), there exists \(j\) such that \(h_j = \mu^{-1}\). As in the proof of Lemma \ref{lem:semigroup_equivalence}, this means we can write \[r_t^{k_t}= (\mu^{-1})^{a}b\] for some \(a\in \pi_1(M)\) and \(b\in N(\mu^k\lambda,\mu^{-1})\).

    Therefore as a final step, we use the equation \(id = r_ts\) to deduce that \(id = r_t^{k_t} s^{k_t}\), and substitute to find
    \[ id = (\mu^{-1})^{a}bs^{k_t}
    \]
    from which we get
    \[
        \mu = (bs^{k_t})^{a^{-1}} \in T(\mu^k \lambda, \mu^{-1}).
    \]

\end{proof}

Lemmas \ref{lem:root_closed} and \ref{lem:semigroup_equivalence} allow  us to re-cast some of our definitions.

\begin{proposition}
    Suppose \(K\) is a nontrivial knot in \(S^3\) and \(g(K)\) is the Seifert genus of \(K\).
    \begin{enumerate}
        \item The knot \(K\) has the conjugate-slope property at \(n\) if and only if \(id \in N(\mu^n \lambda, \mu^{-1})\).
        \item  The knot \(K\) satisfies Property (D) part (b) if and only if \(id \in T(\mu^{2g(K)-1}\lambda, \mu^{-1})\).
    \end{enumerate}
\end{proposition}

\begin{corollary}
    \label{cor:equiv}
    The conjugate-slope property at \(2g(K)-1\) is equivalent to Property (D) part (b) if and only if \(id \in T(\mu^{2g(K)-1}\lambda, \mu^{-1})\) is equivalent to \(id \in N(\mu^{2g(K)-1} \lambda, \mu^{-1})\).
\end{corollary}

We believe that the properties in Corollary \ref{cor:equiv} are equivalent.  However, it should be noted that \(id \in N(S)\) and \(id \in T(S)\) are not equivalent in general, even when \(S\) is a subset of the peripheral subgroup of a knot group.  We show this in the next example.

\begin{example}
    Consider the trefoil knot group, which is isomorphic to the braid group
    \[B_3 = \langle \sigma_1, \sigma_2 \mid \sigma_1 \sigma_2 \sigma_1 = \sigma_2 \sigma_1 \sigma_2 \rangle.
    \]
    The center of \(B_3\) is infinite cyclic, generated by \(\Delta^2 = (\sigma_1 \sigma_2 \sigma_1)^2\), with this element corresponding to the homotopy class of the regular fibre in the Seifert fibering of the trefoil complement.  It is therefore a peripheral element.  A meridian for the trefoil is represented by \(\sigma_1\).

    Suppose that \(id \in N(\Delta^2, \sigma_1^{-6})\).  Using that \(\Delta^2\) is central, we conclude that there exist \(h_1, \ldots, h_k \in B_3\) and \(\ell >0\) such that
    \[ \prod_{i=1}^k h_i \sigma_1^6 h_i^{-1} = \Delta^{2 \ell}.
    \]
    There is a quotient map
    \[ q: B_3 \rightarrow B_3/\langle \Delta^2 \rangle \cong \mathrm{PSL}(2, \mathbb{Z}), \text{ with } \sigma_1^6 \mapsto
        \begin{bmatrix}
            1 & 6 \\
            0 & 1 \\
        \end{bmatrix},
    \]
    and applying this quotient map, one concludes that \(\begin{bmatrix}
        1 & 6 \\
        0 & 1 \\
    \end{bmatrix}\) is generalised torsion in \(\mathrm{PSL}(2, \mathbb{Z})\).  This contradicts \cite[Proposition 4.1]{CHQpreprint}, where they show that there exists a  product of conjugates of \(\begin{bmatrix}
        1 & k \\
        0 & 1 \\
    \end{bmatrix} \in \mathrm{PSL}(2, \mathbb{Z})\) that is equal to the identity if and only if \(|k| \leq 5\). Therefore, \(id \notin N(\Delta^2, \sigma^{-6})\).

    On the other hand, \(T(\Delta^2, \sigma_1^{-6})\) contains both \(\sigma_1 \sigma_2 \sigma_1\) and \(\sigma_1^{-1}\), and therefore it contains \(\sigma_2\). However \(\sigma_2\) is conjugate to \(\sigma_1\), so \(\sigma_1 \in T(\Delta^2, \sigma_1^{-6})\). It follows that \(id \in T(\Delta^2, \sigma_1^{-6})\).  This shows that the conditions \(id \in N(\Delta^2, \sigma^{-6})\) and \(id \in T(\Delta^2, \sigma_1^{-6})\) are indeed different. 
\end{example}

Next, let us consider the connections between the conjugate-slope property and Nie's Property (D) part (a).  For this, we need to briefly review properties of cofinal elements and introduce the cofinal detection property.

\begin{lemma}
    \label{lem:cofinal_products}
    Suppose that \(G\) is a countable group with a left-ordering \(\mathfrak{o}\). If \(g \in G\) is \(\mathfrak{o}\)-cofinal, then every element of \(T(g)\) is \(\mathfrak{o}\)-cofinal and has the same sign as \(g\) relative to the ordering \(\mathfrak{o}\).
\end{lemma}
\begin{proof}
    We consider the case where \(g >_\mathfrak{o}id\), the case of \(g<_\mathfrak{o}id\) being identical.

    By Lemma \ref{TS_structure}, it suffices to show that every element of \(N(S)\) and \(R(S)\) is positive and \(\mathfrak{o}\)-cofinal whenever all elements of \(S\) are positive and \(\mathfrak{o}\)-cofinal. To this end, suppose \(S\) is a collection of positive, cofinal elements. Let \(\rho:G \rightarrow \mathrm{Homeo}_+(\mathbb{R})\) denote the dynamic realisation of the ordering \(\mathfrak{o}\). (For background on dynamic realisations and cofinality, see for example \cite{Navas10, GOD, CR16, BC23}).    One can check that \(h \in S\) is positive and \(\mathfrak{o}\)-cofinal if and only if \(\rho(h)(x)> x\) for all \(x \in \mathbb{R}\).  Thus, if \(h\) is \(\mathfrak{o}\)-cofinal, positive, and \(f \in G\), then
    \[ \rho(f^{-1}hf)(x) > \rho(f^{-1})(\rho(f(x))) = x
    \]
    for all \(x \in \mathbb{R}\), so that the conjugate of an arbitrary positive \(\mathfrak{o}\)-cofinal element is again positive and \(\mathfrak{o}\)-cofinal.  Similarly, if both \(f\) and \(h\) are positive and \(\mathfrak{o}\)-cofinal, then for all \(x \in \mathbb{R}\) we have
    \[\rho(fh)(x) > \rho(f)(x) >x,
    \]
    so that the product \(fh\) is again positive and \(\mathfrak{o}\)-cofinal.  It follows that if \(S\) is a collection of positive and \(\mathfrak{o}\)-cofinal elements, then so is \(N(S)\).

    Likewise, if \(h \in S\) satisfies \(\rho(h)(x) >x\) for all \(x \in \mathbb{R}\) and \(f^n = h\) for some \(n>0\), then \(\rho(f)(x)>x\) for all \(x \in \mathbb{R}\) as well.  So whenever \(S\) is a collection of positive \(\mathfrak{o}\)-cofinal elements, so is \(R(S)\).  This proves the lemma.
\end{proof}

\begin{proposition}
    \label{prop:cofinal_not_allowed}
    If \(K\) is a knot that has the conjugate-slope property at \(n\), then no slope in \((n, \infty)\) is weakly order-detected by a boundary-cofinal left-ordering.
\end{proposition}
\begin{proof}
    Suppose that \(K\) is a knot in \(S^3\) with complement \(M\), and that \(K\) has the conjugate-slope property at \(n\).  Then we may write
    \[ \prod_{i=1}^k h_i^{-1}(\mu^n \lambda)h_i = \prod_{j=1}^\ell g_j^{-1}\mu g_j  \tag{\textasteriskcentered}
    \]
    for some \(k, \ell \in \mathbb{N}\), and \(h_1, \ldots, h_k, g_1, \ldots, g_\ell \in \pi_1(M)\).

    Now suppose that \(\mathfrak{o}\) is a left-ordering of \(\pi_1(M)\) such that \(s(\mathfrak{o}) \in (n, \infty)\).  Then relative to the left-ordering \(\mathfrak{o}\), the elements \(\mu\) and \(\mu^n \lambda\) have opposite signs.  It follows that the two sides of equation \((*)\) represent elements having opposite sign, by Lemma \ref{lem:cofinal_products}.  This is a contradiction.
\end{proof}

\begin{corollary}
    \label{cor:conj_slope_bound}
    Suppose that \(K\) is a knot in \(S^3\) with genus \(g\).  If \(K\) has the conjugate-slope property at \(n\), then \(n \geq 2g-1\).
\end{corollary}
\begin{proof}
    By \cite[Proof of Corollary 1.4]{BGHpreprint}, for every nontrivial knot \(K\) in \(S^3\), the slope \(2g(K)-1\) is detected by a boundary-cofinal left-ordering.  The conclusion of the corollary then follows from an application of Proposition \ref{prop:cofinal_not_allowed}.
\end{proof}

If a knot \(K\) in \(S^3\) with complement \(M\) satisfies Property (D) part (a), then all left-orderings of \(\pi_1(M)\) are boundary-cofinal.  We propose to combine the conjugate-slope property with the following weaker condition, so that together our two conditions are a weakening of Property (D).

\begin{definition}
    \label{def:cofinal}
    A knot \(K\) in \(S^3\) is said to have the \emph{cofinal detection property} if every slope that is weakly order-detected is detected by a boundary-cofinal left-ordering of \(\pi_1(M)\).
\end{definition}

By \cite[Theorem 1.7]{BC23}, every knot \(K\) in \(S^3\) with complement \(M\) for which the slope map \(s : \mathrm{LO}(\pi_1(M)) \rightarrow \mathcal{S}(M)\) is not surjective satisfies Property (D) part (a), and so has the cofinal detection property.  However, the cofinal detection property is a strictly weaker property than Property (D) part (a), as the next example shows.

\begin{example}
    In this example, we show that the knot \(8_{18}\) has the cofinal detection property, but does not satisfy Property (D) part (a).  We do this by showing that every slope is weakly order-detected by a boundary cofinal left-ordering, and a non-boundary-cofinal left-ordering.

    Let \(M_1\) denote the complement of the trefoil knot. We begin by noting that the trefoil knot satisfies Property (D) part (a): from \cite[Proposition 27]{CW12} it follows that the slope map \(s: \mathrm{LO}(\pi_1(M_1)) \rightarrow \mathcal{S}(M_1)\) is not surjective, and so by \cite[Theorem 1.7]{BC23} every left-ordering of \(\pi_1(M_1)\) is boundary-cofinal.  The image of the slope map is precisely \([-\infty, 1]\) \cite[Example 8.4]{BC23}. 

    Let \(M_3\) denote the complement of the knot \(8_{18}\).  According to \cite{KS05}, there is a surjective homomorphism \(\psi: \pi_1(M_3) \rightarrow \pi_1(M_1)\), and one can check that there are also choices of meridian and longitude \(\mu_i, \lambda_i \in \pi_1(M_i)\) for \(i =1, 3\) such that \(\psi(\mu_3) = \mu_1\) and \(\psi(\lambda_3) = \lambda_1\).  Therefore, using the short exact sequence
    \[
        1 \longrightarrow \ker(\psi) \longrightarrow \pi_1(M_3) \stackrel{\psi}{\longrightarrow} \pi_1(M_1) \longrightarrow 1
    \]
    and the observations of the previous paragraph, one can lexicographically create left-orderings of \(\pi_1(M_3)\) that are boundary cofinal, and which detect any given slope in \([-\infty, 1]\).  However, the knot \(8_{18}\) is also amphichiral, so there is an automorphism \(\xi: \pi_1(M_3) \rightarrow \pi_1(M_3)\) with \(\xi(\mu_3) = \mu_3\) and \(\xi(\lambda_3) = \lambda_3^{-1}\).  From this, if \(\mathfrak{o} \in \mathrm{LO}(\pi_1(M_3))\) satisfies \(s(\mathfrak{o}) = [\mu^p \lambda^q]\), then the left-ordering \(\mathfrak{o}'\) with positive cone \(P(\mathfrak{o}') := \xi(P(\mathfrak{o}))\) satisfies \(s(\mathfrak{o}') = [\mu^p \lambda^{-q}]\).  Consequently we know that all slopes in \(\mathcal{S}(M_3)\) are detected by boundary-cofinal left-orderings, so we conclude that \(8_{18}\) has the cofinal detection property.

    We next show that \(8_{18}\) does not satisfy Property (D) part (a).  Let \(M_2\) denote the complement of the figure eight knot.  By \cite[Example 5.7]{BC23}, for every slope \([\alpha] \in \mathcal{S}(M_2)\) there exists a left-ordering \(\mathfrak{o}\) relative to which the peripheral subgroup \(\pi_1(\partial M_2)\) is bounded and satisfying \(s(\mathfrak{o}) = [\alpha]\).
    By \cite{KS05} there exists a homomorphism \(\phi: \pi_1(M_3) \rightarrow \pi_1(M_2)\) and choices of meridian and longitude \(\mu_i, \lambda_i \in \pi_1(M_i)\) for \(i =2, 3\) such that \(\phi(\mu_3) = \mu_2\) and \(\phi(\lambda_3) = \lambda_2\).  Using the short exact sequence
    \[
        1 \longrightarrow \ker(\phi) \longrightarrow \pi_1(M_3) \stackrel{\phi}{\longrightarrow} \pi_1(M_2) \longrightarrow 1
    \]
    and the observations of the previous paragraph, one can lexicographically construct left-orderings of \(\pi_1(M_3)\) to show that every slope in \(\mathcal{S}(M_3)\) is weakly order-detected by a left-ordering that is not boundary cofinal.  In particular, \(8_{18}\) does not satisfy Property (D) part (a).

\end{example}




Before connecting the conjugate-slope property and the cofinal detection property with Property (D), we note that there is also a second property appearing in the literature that was introduced for the purpose of studying non-left-orderability of Dehn fillings.

\begin{definition}[\cite{CW11}]
    Let \(r \in \mathbb{Q}\) be a positive rational number and \(M\) the complement of \(K\).  Set \(S_r = \{ \mu^p \lambda^q \mid \frac{p}{q} \geq r \}\).  We say that \(K\) is \emph{\(r\)-decayed} if, for every positive cone \(P \subset \pi_1(M)\), whenever \(P \cap S_r \neq \emptyset\) then \(S_r \subset P\).
\end{definition}

With all of the required definitions in hand, we have the following relationships.

\begin{theorem}
    \label{thm:not_weak_detection}
    Let \(M\) denote the complement of a nontrivial knot \(K\) in \(S^3\).  Then the following implications hold.

    \begin{enumerate}
        \item If \(K\) has the conjugate-slope property at \(n\) and the cofinal detection property, then \(K\) is \(r\)-decayed for all \(r >n\).
        \item The knot \(K\) is \(r\)-decayed if and only if no slope in \([r, \infty)\) is weakly order-detected.
    \end{enumerate}
\end{theorem}
\begin{proof}
    We first prove \((1)\).  To this end, let \(P \subset \pi_1(M)\) be a positive cone with associated left-ordering \(\mathfrak{o}\), and recall that \(S_r = \{ \mu^p \lambda^q \mid \frac{p}{q} \geq r \}\) for all positive \(r \in \mathbb{Q}\).  Suppose that \(r>n\) and that  \(P \cap S_r\) is nonempty and \(S_r \not \subset P\).  Then we may choose \(\mu^{p_1}\lambda^{q_1}, \mu^{p_2} \lambda^{q_2} \in S_r\) such that \(\mu^{p_1}\lambda^{q_1} <_\mathfrak{o} id <_\mathfrak{o} \mu^{p_2} \lambda^{q_2}\), without loss of generality we may assume that \(r \leq \frac{p_1}{q_1} < \frac{p_2}{q_2}\).  Then \(s(\mathfrak{o}) \in [\frac{p_1}{q_1}, \frac{p_2}{q_2}] \subset [r, \infty)\).  By assumption, there exists a boundary-cofinal ordering \(\mathfrak{o}'\) of \(\pi_1(M)\) with \(s(\mathfrak{o}') = s(\mathfrak{o}) \in [r, \infty) \subset (n, \infty)\), which contradicts Proposition \ref{prop:cofinal_not_allowed}.

    To prove claim (2), we first follow \cite[Proof of Theorem 3]{CW11}, and \cite[Proof of Theorem 9]{CW12} to deal with slopes in \((r, \infty)\).  Suppose \(K\) is \(r\)-decayed and that \(\mathfrak{o}\) is an ordering of \(\pi_1(M)\) with \(s(\mathfrak{o}) \in (r, \infty)\).  Then we may choose \(\mu^{p_1}\lambda^{q_1}, \mu^{p_2} \lambda^{q_2} \in S_r\) such that \(\mu^{p_1}\lambda^{q_1} <_\mathfrak{o} id <_\mathfrak{o} \mu^{p_2} \lambda^{q_2}\), so that \(P(\mathfrak{o}) \cap S_r\) is nonempty and \(S_r \not \subset P(\mathfrak{o})\), a contradiction. Next suppose that \(\mathfrak{o}\) is an ordering of \(\pi_1(M)\) with \(s(\mathfrak{o}) = [\mu^p \lambda^q]\) and set \(r = p/q\).  By \cite[Proposition 5.2]{BC23}, there exists a relatively convex subgroup \(C \subset \pi_1(M)\) with \(C \cap \pi_1(\partial M) = \langle \mu^p \lambda^q \rangle\).  Using this convex subgroup, we can lexicographically construct a left-ordering \(\mathfrak{o}'\) of \(\pi_1(M)\) such that \(s(\mathfrak{o}') = [\mu^p \lambda^q]\) and \(P(\mathfrak{o}') \cap S_r  = \{ (\mu^p \lambda^q)^k \mid k>0\}\), a contradiction.

    Conversely, if no slope in \([r, \infty)\) is weakly order-detected, then in every ordering \(\mathfrak{o}\) of \(\pi_1(M)\) the elements of \(S_r\) must have the same sign. Correspondingly, either \(S_r \subset P(\mathfrak{o})\) or \(S_r \subset P(\mathfrak{o})^{-1}\), implying that \(K\) is \(r\)-decayed.
\end{proof}

\begin{theorem}
    \label{thm:con_slope_to_D}
    For a nontrivial knot \(K\) in \(S^3\), the following implications hold:
    \begin{enumerate}
        \item If \(K\) has the conjugate-slope property at \(2g(K)-1\) and the cofinal detection property, then \(K\) has Property (D).
        \item If \(K\) has property (D), then \(K\) is \(r\)-decayed for every \(r> 2g(K)-1\).
    \end{enumerate}

\end{theorem}
\begin{proof}
    To show (1), we first note that part (b) of Definition \ref{propertyD} follows automatically from the assumption that \(K\) has the conjugate-slope property at \(2g(K)-1\).  To show part (a) of Definition \ref{propertyD}, suppose that \(\mathfrak{o}\) is a left-ordering of \(\pi_1(M)\) with \(s(\mathfrak{o}) \in (2g(K)-1, \infty)\).  By assumption there exists \(\mathfrak{o}'
    \) that is boundary-cofinal with \(s(\mathfrak{o}') = s(\mathfrak{o})\), which is not possible by Proposition \ref{prop:cofinal_not_allowed}.  Therefore the slope map \(s : \mathrm{LO}(\pi_1(M)) \rightarrow \mathcal{S}(M)\) is not surjective, and so every left-ordering of \(\pi_1(M)\) is boundary-cofinal by \cite[Theorem 1.7]{BC23}.  As already noted, part (a) of Definition \ref{propertyD} is equivalent to the claim that every left-ordering of \(\pi_1(M)\) is boundary-cofinal, and so Property (D) holds.

    For (2), since part (a) of Definition \ref{propertyD} implies that every left-ordering of \(\pi_1(M)\) is boundary-cofinal, one can proceed exactly as in the proof of Theorem \ref{thm:not_weak_detection}(1). (\textit{cf.} \cite[Proposition 8.3]{BGHpreprint}).
\end{proof}

We do not know whether or not the converse implications of Theorem \ref{thm:con_slope_to_D} and Theorem \ref{thm:not_weak_detection}(1) hold.

\section{Consequences of the conjugate-slope property}
\label{sec:consequences}

\subsection{Orderability, Dehn fillings and the L-space conjecture}

The relationships outlined in the previous section are intended for use in addressing the left-orderability aspect of the L-space conjecture.  The connection comes from the following observation: Let \(M\) denote the complement of a knot in \(S^3\) and \([\alpha] \in \mathcal{S}(M)\). Use $M(\alpha)$ to denote the result of Dehn filling $M$ along the slope $\alpha$. If \(\pi_1(M(\alpha))\) is left-orderable, then \([\alpha]\) is order-detected, and thus weakly order-detected \cite[Corollary 1.10]{BC23}. Therefore if we are able to show that \([\alpha]\) is not weakly order-detected, then \(\pi_1(M(\alpha))\) is not left-orderable.

In light of this, we have the following theorem.

\begin{theorem}
    \label{thm:nonLO}
    Suppose that \(K\) is a nontrivial knot in \(S^3\) with complement \(M\).  If \(K\) has the conjugate-slope property at \(n\) and the cofinal detection property, then \(\pi_1(M(\mu^p\lambda^q))\) is not left-orderable for all \(p/q \in [n, \infty)\).
\end{theorem}
\begin{proof}
    First, suppose that \(p/q \in (n, \infty)\).  In light of the discussion preceding this proof, if \(\pi_1(M(\mu^p\lambda^q))\) were left-orderable, then \([\mu^p \lambda^q]\) would be weakly order-detected.  But this contradicts Theorem \ref{thm:not_weak_detection}.

    It remains to consider the case where \(p/q =n\).  To this end, suppose that \(\pi_1(M(\mu^n \lambda))\) admits a left-ordering \(\mathfrak{o}\), and that
    \[ \prod_{i=1}^m g_i^{-1}(\mu^n \lambda)g_i = \prod_{j=1}^{nm} h_j^{-1} \mu h_j,
    \]
    and let \(q: \pi_1(M) \rightarrow \pi_1(M)/\langle \langle \mu^n \lambda \rangle \rangle \cong \pi_1(M(\mu^n \lambda))\) denote the quotient map. Suppose that \(q(\mu)\) is \(\mathfrak{o}\)-cofinal.

    Then \(\prod_{j=1}^{nm} q(h_j^{-1}) q(\mu)  q(h_j)\) is \(\mathfrak{o}\)-cofinal by Lemma \ref{lem:cofinal_products}, and yet by the equality above, \[\prod_{j=1}^{nm} q(h_j^{-1}) q(\mu)  q(h_j) = id \in \pi_1(M(\mu^n \lambda)),
    \]
    a contradiction.

    In particular, since \(q(\mu)\) is bounded in the left-ordering \(\mathfrak{o}\), we may use \(\mathfrak{o}\) and the short exact sequence
    \[ 1 \longrightarrow \langle \langle \mu^n \lambda  \rangle \rangle \longrightarrow \pi_1(M) \stackrel{q}{\longrightarrow} \pi_1(M(\mu^n \lambda)) \longrightarrow 1
    \]
    to construct a left-ordering of \(\pi_1(M)\) in which \(\pi_1(\partial M)\) is bounded.  But since Theorem \ref{thm:not_weak_detection} says no slope in \((n, \infty)\) is weakly order-detected, \cite[Theorem 1.7]{BC23} implies that every left-ordering of \(\pi_1(M)\) must be boundary cofinal.  This contradiction finishes the proof.
\end{proof}

From this theorem, a plausible approach to understanding non-left-orderability of fundamental groups arising from Dehn fillings emerges.  The knots which enjoy both the conjugate-slope property and the cofinal detection property will have many Dehn fillings with non-left-orderable fundamental group, and so we conjecture:

\begin{conjecture}
    A nontrivial knot \(K\) is an L-space knot if and only if it has both the conjugate-slope property and the cofinal detection property.
\end{conjecture}

\subsection{Generalised torsion and Dehn fillings}

Recall that if \(G\) is a group, then \(g \in G\) is called \emph{generalised torsion} if there exist \(h_1, \dots, h_n \in G\) such that
\[ \prod_{i=1}^n h_i^{-1}g h_i = id.
\]
The significance of generalised torsion in \(3\)-manifold groups was already noted in \cite{KT17}, where they conjectured that for a \(3\)-manifold \(M\), \(\pi_1(M)\) is bi-orderable if and only if it is generalised torsion-free.  This conjecture was investigated in the context of manifolds arising from Dehn surgery along knots in \(S^3\) \cite{IKT21}.  While their conjecture has since been disproved \cite[Section 5]{CC25}, understanding when the meridian becomes generalised torsion in the Dehn-filled manifold appears to be connected, via the conjugate-slope property, to the L-space conjecture.

\begin{theorem}
    \label{thm:gen_torsion}
    Suppose that \(M\) is the exterior of a knot in \(S^3\), and let \(\{\mu, \lambda\}\) denote a choice of meridian and longitude basis for the peripheral subgroup of \(\pi_1(M)\).  Suppose that either:
    \begin{enumerate}
        \item There exist \(g_1, \ldots, g_k \in \pi_1(M)\) and \(n>0\) such that \(\mu^n \lambda = \prod_{i=1}^k g_i \mu g_i^{-1}\), or
        \item There exist \(g_1, \ldots, g_k, h \in \pi_1(M)\) and \(n>0\) such that \(\mu^n \lambda (h \mu^n \lambda h^{-1}) = \prod_{i=1}^k g_i \mu g_i^{-1}\).
    \end{enumerate}
    Then for all rational slopes \(\alpha \in [n, \infty)\), the image of the meridian is generalised torsion in \(\pi_1(M(\alpha))\).
\end{theorem}
\begin{proof}
    Let \(n>0\) and first suppose that \(\mu^n \lambda = \prod_{i=1}^k g_i \mu g_i^{-1}\).  Given \(p/q \in [n, \infty)\) with \(p, q >0\) and relatively prime, note that
    \[
        \mu^p \lambda^q = (\mu^n \lambda)^q \mu^{p-nq} = (\prod_{i=1}^k g_i \mu g_i^{-1})^q\mu^{p-nq}.
    \]
    It follows that the image of \(\mu\) is generalised torsion in the quotient \(\pi_1(M(\mu^p \lambda^q)) \cong \pi_1(M)/
    \langle \langle \mu^p \lambda^q \rangle \rangle\).

    On the other hand, suppose that \(\mu^n \lambda (h \mu^n \lambda h^{-1}) = \prod_{i=1}^k g_i \mu g_i^{-1}\).  Set
    \[ P = (h \mu^n \lambda h^{-1})\prod_{i=1}^k g_i \mu g_i^{-1}(h \mu^n \lambda h^{-1})^{-1},
    \]
    which is a product of conjugates of \(\mu\),
    and note that \((h \mu^n \lambda h^{-1})\mu^n \lambda = P\). Therefore
    \[ (h \mu^n \lambda h^{-1})\mu^{2n} \lambda^2 (h \mu^n \lambda h^{-1}) = P \prod_{i=1}^k g_i \mu g_i^{-1}
    \] and thus
    \[ \mu^{2n} \lambda^2 (h \mu^{2n} \lambda^2 h^{-1}) = (h \mu^n \lambda h^{-1})^{-1}P \prod_{i=1}^k g_i \mu g_i^{-1}(h \mu^n \lambda h^{-1}),
    \]
    note that the right-hand side is a product of conjugates of \(\mu\).  In general, one can apply induction: suppose that \(\mu^{n \ell} \lambda^{\ell} (h \mu^{n \ell} \lambda^{ \ell} h^{-1}) = Q\) is a product of conjugates of \(\mu\) for \(\ell >0\), then
    \[   (h \mu^n \lambda h^{-1})\mu^{n(\ell + 1)} \lambda^{\ell+1} (h \mu^{n \ell} \lambda^{ \ell} h^{-1})= PQ
    \]
    so that
    \[\mu^{n(\ell + 1)} \lambda^{\ell+1} (h \mu^{n (\ell+1)} \lambda^{ \ell+1} h^{-1}) = (h \mu^n \lambda h^{-1})^{-1}PQ(h \mu^n \lambda h^{-1}).
    \]
    By induction we conclude that \(\mu^{n \ell} \lambda^{\ell} (h \mu^{n \ell} \lambda^{ \ell} h^{-1})\) can be written as a product of conjugates of \(\mu\) for all \(\ell >0\).

    Now let  \(p/q \in [n, \infty)\) with \(p, q >0\) and relatively prime.  Then write
    \[\mu^{n q} \lambda^{q} (h \mu^{n q} \lambda^{ q} h^{-1}) = R
    \]
    where \(R\) is a product of conjugates of \(\mu\), and note that
    \[ \mu^{p} \lambda^{q} (h \mu^{p} \lambda^{q}h^{-1}) =
        \mu^{p-nq}(\mu^{n q} \lambda^{q} (h \mu^{n q} \lambda^{ q} h^{-1}) )(h \mu^{p-nq} h^{-1}) = \mu^{p-nq}R(h \mu^{p-nq} h^{-1}).
    \]
    Since the left-hand side is trivial in \(\pi_1(M(\mu^p \lambda^q))\), we conclude from the right-hand side that the image of \(\mu\) is generalised torsion in \(\pi_1(M(\mu^p \lambda^q))\).
\end{proof}

In particular, Theorem \ref{thm:gen_torsion} applies to all of the knots that have the conjugate-slope property as a consequence of Theorem \ref{thm:con_slope_property} or Propositions \ref{prop:one_bridge} or \ref{prop:two_bridges}.  We suspect that generalised torsion is connected to the conjugate-slope property in a more general setting, and offer the next observation as evidence.

\begin{proposition}
    Suppose that \(M\) is the exterior of a knot \(K\) in \(S^3\), and that \(K\) has the conjugate-slope property at \(n\).  Then for all \(m>n\), \(m \in \mathbb{N}\), the image of the meridian is generalised torsion in \(\pi_1(M(\mu^m \lambda))\).
\end{proposition}
\begin{proof}
    Suppose that \(\prod_{i=1}^k g_i^{-1} \mu^n \lambda g_i = \prod_{j=1}^\ell h_j^{-1} \mu h_j\) and let \(\phi: \pi_1(M) \rightarrow \pi_1(M(\mu^m \lambda))\) denote the quotient map.  Applying this map to both sides of the given equation, and using \(\phi(\mu^m \lambda) = id\), we arrive at
    \[
        \prod_{i=1}^k \phi(g_i)^{-1} \phi(\mu^{n-m})  \phi(g_i)  =  \prod_{j=1}^\ell \phi(h_j)^{-1} \phi(\mu) \phi(h_j).
    \]
    Bringing all of the terms on the left-hand side over to the right, we arrive at the conclusion.
\end{proof}

It is also not difficult to see that if \(K\) has the conjugate-slope property at \(n\), the image of \(\mu\) is generalised torsion in \(\pi_1(M(\mu^p \lambda))\) for all \(p>n\).  A natural generalisation is the following question.

\begin{question}
    If \(K\) has the conjugate-slope property at \(n\), is the image of \(\mu\) generalised torsion in \(\pi_1(M(\mu^p \lambda^q))\) for all \(p/q>n\)?
\end{question}

\bibliographystyle{alpha} 
\bibliography{idea.bib} 
\end{document}